\documentclass[a4paper]{article}%
\usepackage{amsmath}
\usepackage{amsfonts}
\usepackage{amssymb}
\usepackage{graphicx}
\usepackage{hyperref}
\usepackage{textcomp}%
\providecommand{\U}[1]{\protect\rule{.1in}{.1in}}
\newtheorem{theorem}{Theorem}[section]

\newtheorem{corollary}[theorem]{Corollary}

\newtheorem{lemma}[theorem]{Lemma}

\newtheorem{proposition}[theorem]{Proposition}
\newtheorem{remark}[theorem]{Remark}

\numberwithin{equation}{section}
\newenvironment{proof}[1][Proof]{\noindent\textbf{#1.} }{\ \rule{0.5em}{0.5em}}
\hypersetup{
colorlinks   = true,   urlcolor     = blue,   linkcolor    = blue,   citecolor   = red }
\begin{document}

\title{The limiting behavior of solutions to $p$-Laplacian problems with convection
and exponential terms}
\author{Anderson L. A. de Araujo\\{\small Universidade Federal de Vi\c{c}osa}\\{\small Vi\c{c}osa, MG, 36.570-900, Brazil}\\{\small anderson.araujo@ufv.br\medskip}\\Grey Ercole\\{\small Universidade Federal de Minas Gerais}\\{\small Belo Horizonte, MG, 30.123-970, Brazil}\\{\small grey@mat.ufmg.br\medskip}\\Julio C. Lanazca Vargas\\{\small Universidade Federal de Minas Gerais}\\{\small Belo Horizonte, MG, 30.123-970, Brazil}\\{\small jlanazca@gmail.com\medskip}}
\maketitle

\begin{abstract}
\noindent We consider, for $a,l\geq1,$ $b,s,\alpha>0,$ and $p>q\geq1,$ the
homogeneous Dirichlet problem for the equation $-\Delta_{p}u=\lambda
u^{q-1}+\beta u^{a-1}\left\vert \nabla u\right\vert ^{b}+mu^{l-1}e^{\alpha
u^{s}}$ in a smooth bounded domain $\Omega\subset\mathbb{R}^{N}.$ We prove
that under certain setting of the parameters $\lambda,$ $\beta$ and $m$ the
problem admits at least one positive solution. Using this result we prove that
if $\lambda,\beta>0$ are arbitrarily fixed and $m$ is sufficiently small, then
the problem has a positive solution $u_{p},$ for all $p$ sufficiently large.
In addition, we show that $u_{p}$ converges uniformly to the distance function
to the boundary of $\Omega,$ as $p\rightarrow\infty.$ This convergence result
is new for nonlinearities involving a convection term.

\bigskip


\noindent\textbf{keywords}{\small \textbf{:} Convection term, distance
function, exponential term, gradient estimate.}

\end{abstract}

\section{Introduction}

In this paper we consider the Dirichlet problem
\begin{equation}
\left\{
\begin{array}
[c]{lll}%
-\Delta_{p}u=\lambda u^{q-1}+\beta u^{a-1}\left\vert \nabla u\right\vert
^{b}+mu^{l-1}e^{\alpha u^{s}} & \text{in} & \Omega\\
u>0 & \text{in} & \Omega\\
u=0 & \text{on} & \partial\Omega
\end{array}
\right.  \tag{P}\label{P}%
\end{equation}
where $\Omega$ is a smooth bounded domain $\mathbb{R}^{N},$ $N\geq2,$ and
$\Delta_{p}u:=\operatorname{div}(\left\vert \nabla u\right\vert ^{p-2}\nabla
u)$ is the $p$-Laplacian operator, $p>1$. The parameters $\lambda,$ $\beta$
and $m$ are nonnegative and the constants $q,a,b,l,\alpha$ and $s$ satisfy
\[
a,l\geq1,\text{ \ }b>0,\text{ \ }s,\alpha\geq0,\text{ \ and }p>q\geq1.
\]

Our main result in this paper is stated as follows, where $d_{\Omega}$ denotes
the distance function to the boundary:
\[
d_{\Omega}(x):=\min_{y\in\partial\Omega}\left\vert x-y\right\vert ,\text{
\ }x\in\overline{\Omega}.
\]

\begin{theorem}
\label{main}Assume that $\partial\Omega\in C^{1,1}.$ Let $\lambda$ and $\beta$
be arbitrary, but fixed positive real numbers, and let $m$ be such that
\begin{equation}
0<m<m_{\infty}:=(\lambda\left\Vert d_{\Omega}\right\Vert _{\infty}^{q-1}%
+\beta\left\Vert d_{\Omega}\right\Vert _{\infty}^{a-1})\left\Vert d_{\Omega
}\right\Vert _{\infty}^{l-1}e^{-\alpha\left\Vert d_{\Omega}\right\Vert
_{\infty}^{s}}. \label{minf}%
\end{equation}

There exists $p_{0}>\max\left\{  q,a+b,l\right\}  $ such that if $p>p_{0}$
then the Dirichlet problem (\ref{P}) admits a weak solution $u_{p}\in
W_{0}^{1,p}(\Omega).$ Moreover,%
\begin{equation}
\lim_{p\rightarrow\infty}u_{p}=d_{\Omega}\text{ uniformly in }\overline
{\Omega}. \label{upd}%
\end{equation}

\end{theorem}

In the last decades, this kind of limiting behavior of solutions to Dirichlet
problems of the form%
\[
\left\{
\begin{array}
[c]{rrll}%
-\operatorname{div}(\phi_{p}(\left\vert \nabla u\right\vert )\nabla u) & = &
f(x,u) & \text{in }\Omega\\
u & > & 0 & \text{in }\Omega,\\
u & = & 0 & \text{on }\partial\Omega
\end{array}
\right.
\]
has been obtained by many authors. We refer to \cite{BDM}, \cite{K1},
\cite{Mi18} for the $p$-Laplacian, i.e. $\phi_{p}(t)=t^{p-2}$, and
\cite{Bo17}, \cite{GE1}, \cite{EFMP}, \cite{FM19}, \cite{FMD17}, \cite{GD21}
for more general functions $\phi_{p}.$

However, up to our acknowledge, this is the first work dealing with the
limiting behavior of solutions $u_{p}$ to a $p$-family of Dirichlet problems
with convection (i.e. gradient) terms.

The solution $u_{p}$ of Theorem \ref{main} is obtained as an application of
Theorem \ref{sbSP} stated below. To properly state this existence result let
us fix the notation, give some definitions and recall some facts.

The standard norm of Lebesgue space $L^{t}(\Omega),$ $1\leq t\leq\infty,$ will
be denoted by $\left\Vert \cdot\right\Vert _{t}.$

We will denote by $\phi_{p}$ the $p$-torsion function associated with
$\Omega,$ that is, the weak solution in the Sobolev space $W_{0}^{1,p}%
(\Omega)$ to the $p$-torsional creep problem%
\[
\left\{
\begin{array}
[c]{rrll}%
-\Delta_{p}v & = & 1 & \text{in }\Omega\\
v & = & 0 & \text{on }\partial\Omega.
\end{array}
\right.
\]

The first eigenvalue of the Dirichlet $p$-Laplacian will be denoted by
$\lambda_{p}$ whereas $e_{p}$ will denote the positive and $L^{\infty}%
$-normalized eigenfunction corresponding to $\lambda_{p}$ (so that, $e_{p}>0$
in $\Omega$ and $\left\Vert e_{p}\right\Vert _{\infty}=1$). We recall that
\[
\lambda_{p}:=\inf\left\{  \frac{\left\Vert \nabla v\right\Vert _{p}^{p}%
}{\left\Vert v\right\Vert _{p}^{p}}:v\in W_{0}^{1,p}(\Omega)\setminus\left\{
0\right\}  \right\}  =\frac{\left\Vert \nabla e_{p}\right\Vert _{p}^{p}%
}{\left\Vert e_{p}\right\Vert _{p}^{p}}%
\]
and also that $e_{p}$ is a weak solution to the Dirichlet problem
\[
\left\{
\begin{array}
[c]{rrll}%
-\Delta_{p}v & = & \lambda_{p}\left\vert v\right\vert ^{p-2}v & \text{in
}\Omega\\
v & = & 0 & \text{on }\partial\Omega.
\end{array}
\right.
\]

Let us define
\begin{equation}
k_{p}:=\sup_{w\in\mathcal{S}_{p}}\left\Vert \nabla w\right\Vert _{\infty}
\label{kp}%
\end{equation}
where
\begin{equation}
\mathcal{S}_{p}:=\left\{  w\in W_{0}^{1,p}(\Omega):-\Delta_{p}w=g\text{ in
}\Omega,\text{ for some }g\in L^{\infty}(\Omega),\text{ }\left\Vert
g\right\Vert _{\infty}=1\right\}  . \label{Sp}%
\end{equation}
(Note that $\phi_{p}\in\mathcal{S}_{p}$.)

For each $M>0$ let $\mathcal{E}(M)$ denote the region of $\mathbb{R}_{+}^{3}$
defined by%
\[
\mathcal{E}(M):=\left\{  (\lambda,\beta,m)\in\mathbb{R}_{+}^{3}:\frac{\lambda
A_{p}}{M^{p-q}}+\frac{\beta B_{p}}{M^{p-(a+b)}}+\frac{mA_{p}}{M^{p-l}%
e^{-\alpha M^{s}}}\leq1\right\}  ,
\]
where
\begin{equation}
A_{p}:=\left\Vert \phi_{p}\right\Vert _{\infty}^{p-1}\text{ \ and \ }%
B_{p}:=k_{p}^{b}\left\Vert \phi_{p}\right\Vert _{\infty}^{p-1-b}. \label{ApBp}%
\end{equation}

Now, we can state our main existence result.

\begin{theorem}
\label{sbSP}Assume that $\partial\Omega\in C^{1,\tau}$ for some $\tau\in(0,1)$
and suppose that $(\lambda,\beta,m)\in\mathcal{E}(M)$ for some $M>0.$ Then,
the Dirichlet problem (\ref{P}) admits at least one weak solution $u_{p}\in
W_{0}^{1,p}(\Omega)$ satisfying the bounds
\begin{equation}
\left(  \frac{\lambda}{\lambda_{p}}\right)  ^{1/(p-q)}e_{p}\leq u_{p}\leq
\frac{M}{\left\Vert \phi_{p}\right\Vert _{\infty}}\phi_{p}\text{ \ and
\ }\left\Vert \nabla u_{p}\right\Vert _{\infty}\leq\frac{k_{p}M}{\left\Vert
\phi_{p}\right\Vert _{\infty}}. \label{bounds}%
\end{equation}

\end{theorem}

We emphasize that this existence result does not impose any restriction
neither to the exponent $s$ in the exponential term nor to the exponent $b$ in
the convection term, respectively to the critical values $\frac{N}{N-1}$
(Trudinger-Moser inequality) and $p$ (the natural growth of the gradient).

The proof\ of Theorem \ref{sbSP} is given in Section \ref{sec1}. It is
inspired by the approach introduced by Bueno and Ercole \cite{BE}, which
relies on a combination of the sub-super solution method with a version of the
Schauder Fixed Point Theorem.

In \cite{Djairo}, de Figueiredo, Gossez, Quoirin and Ubilla proved existence
results for the following class of $p$-Laplacian problems
\[
\left\{
\begin{array}
[c]{rrll}%
-\Delta_{p}u & = & g(u)\left\vert \nabla u\right\vert ^{p}+f(x,u) & \text{in
}\Omega\\
u & > & 0 & \text{in }\Omega,\\
u & = & 0 & \text{on }\partial\Omega.
\end{array}
\right.
\]

Their results apply to the following particular nonlinearities (in our
notation), where $p^{\ast}$ denotes the well-known critical Sobolev exponent:

\begin{itemize}
\item[(a)] $\beta\left\vert \nabla u\right\vert ^{p}+u^{l-1}e^{\alpha u},$
with $\beta>0,$ $l>p$ and $0<\alpha<\beta\frac{p^{\ast}-p}{p-1}$
(\cite[Example 2.3]{Djairo});

\item[(b)] $u^{a-1}\left\vert \nabla u\right\vert ^{p}+mu^{p-1}e^{\alpha
u^{a}},$ with $a>1,$ $0<\alpha<\frac{p^{\ast}-p}{(p-1)a}$ and $0<m<\lambda
_{p}$ (\cite[Example 2.9]{Djairo});

\item[(c)] $\beta\left\vert \nabla u\right\vert ^{p}+u^{l-1},$ with $\beta>0$
\ and $l>p$ (\cite[Example 2.13]{Djairo}); and

\item[(d)] $\beta\left\vert \nabla u\right\vert ^{p}+mu^{l-1}e^{\alpha u},$
\ with \ $\beta>0,$ $0<\alpha<\beta\frac{p^{\ast}-p}{p-1},$ $1\leq l<p$ \ and
$m$ positive, sufficiently small (\cite[Theorem 2.17]{Djairo}).
\end{itemize}

Our Theorem \ref{sbSP} complements the existence results for these particular
nonlinearities, with a sublinear term $\lambda u^{q-1}$ ($1\leq q<p$) added
and with $\left\vert \nabla u\right\vert ^{b}$ ($b>0$) in the place of
$\left\vert \nabla u\right\vert ^{p}.$ Indeed, items (a) and (c) are
complemented by Corollary \ref{cor2}, item (b) is complemented by Corollary
\ref{cor3}, and item (d) is complemented by Corollary \ref{up}.

In \cite{AM23} de Araujo and Montenegro considered the following Dirichlet
problem (in our notation)%
\[
\left\{
\begin{array}
[c]{rrll}%
-\Delta_{p}u & = & \lambda u^{q-1}+u^{l-1}e^{\alpha u^{s}} & \text{in }%
\Omega\\
u & > & 0 & \text{in }\Omega,\\
u & = & 0 & \text{on }\partial\Omega
\end{array}
\right.
\]
where $1<q<p,$ $l\geq1$ and $s>0.$ Using the approach introduced in \cite{BE}
they proved an existence result for $\lambda$ and $\alpha$ sufficiently small,
by assuming $l\not =p.$ Besides including a convection term, our Theorem
\ref{sbSP} complements Theorem 1.1 of \cite{AM23}.

Still in Section \ref{sec1} we present two more applications of Theorem
\ref{sbSP} (see Corollaries \ref{cor1} and \ref{cor2}) that extend a recent
existence result obtained by de Araujo and Faria in \cite{AL}.

In Section \ref{none} we prove, as consequence of the Picone's inequality, a
nonexistence result for the Dirichlet problem%
\[
\left\{
\begin{array}
[c]{rrll}%
-\Delta_{p}u & = & mu^{l-1}e^{\alpha u^{s}}+g(x) & \text{in }\Omega\\
u & = & 0 & \text{on }\partial\Omega,
\end{array}
\right.
\]
stated in Proposition \ref{bm}, for $g\in L^{\infty}(\Omega),$ $g\geq0,$
$1\leq l<p$ and $\alpha,s>0.$ This result generalizes Theorem 2.1 by Garcia
Alonso and Peral Alonso in \cite{Peral} and also shows that a restriction for
the parameter $m$ in (\ref{P}) is to be expected when $1\leq l<p$.

In Section \ref{sec2} we prove Theorem \ref{main}. To obtain $u_{p}$ we apply
Theorem \ref{sbSP}, and to achieve the limiting behavior (\ref{upd}) we show
that%
\[
\lim_{p\rightarrow\infty}k_{p}=1
\]
where $k_{p}$ is defined in (\ref{kp}). (Recall that $k_{p}$ appears in
(\ref{ApBp}) and (\ref{bounds})). A crucial step in the proof of this limit
comes from the estimate
\[
\left\Vert \nabla w\right\Vert _{\infty}\leq(cp^{\gamma})^{\frac{1}{p-1}%
}\text{\ \ for all }w\in\mathcal{S}_{p},\text{ }p\geq2,
\]
where $c$ and $\gamma$ are positive constants independent of $p$ and $w.$ Such
an estimate is deduced by applying a version of the global gradient estimate
by Cianchi and Maz'ya (see \cite{CM11}) adapted for the $p$-Laplacian by
Ercole in \cite{GE}.

\section{Existence and applications\label{sec1}}

In this section we assume that $\partial\Omega$ is at least of class
$C^{1,\tau}$ for some $\tau\in(0,1).$

We recall that in the particular case where $\Omega$ is a ball centered at
$x_{0}\in\mathbb{R}^{N}$ with radius $R>0$ the function $\phi_{p}$ is radially
symmetric, radially decreasing and explicitly given by the expression%

\[
\phi_{p}(x)=\frac{p-1}{p}N^{-\frac{1}{p-1}}\left(  R^{\frac{p}{p-1}%
}-\left\vert x-x_{0}\right\vert ^{\frac{p}{p-1}}\right)  ,\text{ }\left\vert
x-x_{0}\right\vert \leq R.
\]
It follows from this formula that%
\begin{equation}
\left\Vert \phi_{p}\right\Vert _{\infty}=\frac{p-1}{p}N^{-\frac{1}{p-1}%
}R^{\frac{p}{p-1}}. \label{radmax}%
\end{equation}

As for a general bounded domain $\Omega,$ one can combine Schwarz
symmetrization and (\ref{radmax}) to derive the upper bound
\begin{equation}
\left\Vert \phi_{p}\right\Vert _{\infty}\leq M_{0}:=\frac{p-1}{p}N^{-\frac
{1}{p-1}}\left(  \frac{\left\vert \Omega\right\vert }{\omega_{N}}\right)
^{\frac{p}{N(p-1)}}, \label{bt}%
\end{equation}
where $\left\vert \Omega\right\vert $ denotes the volume of $\Omega$ and
$\omega_{N}$ denotes the volume of the unit ball.

In sequel $\left\vert \cdot\right\vert _{1,\eta}$ will denote the norm of
$C^{1,\eta}(\overline{\Omega})$ defined by%
\[
\left\vert u\right\vert _{1,\eta}:=\left\Vert u\right\Vert _{\infty
}+\left\Vert \nabla u\right\Vert _{\infty}+\sum_{i=1}^{N}\left[  u_{x_{i}%
}\right]  _{\eta}%
\]
where%
\[
\left[  u_{x_{i}}\right]  _{\eta}:=\sup\left\{  \frac{\left\vert u_{x_{i}%
}(x)-u_{x_{i}}(y)\right\vert }{\left\vert x-y\right\vert ^{\eta}}%
:x,y\in\overline{\Omega}\text{ \ and \ }x\not =y\right\}
\]
for each $i\in\left\{  1,2,\ldots,N\right\}  .$

The set $\mathcal{S}_{p}$ that appears in the sequence is defined in (\ref{Sp}).

\begin{proposition}
\label{p1}If $u\in\mathcal{S}_{p},$ then
\begin{equation}
\left\vert u(x)\right\vert \leq\phi_{p}(x)\quad\forall\,x\in\overline{\Omega}.
\label{uinf}%
\end{equation}
Moreover, there exist positive constants $\eta_{p}$ and $K_{p},$ such that
$u\in C^{1,\eta_{p}}(\overline{\Omega})$ and
\[
\left\vert u\right\vert _{1,\eta_{p}}\leq K_{p}.
\]

\end{proposition}

\begin{proof}
Inequality (\ref{uinf}) follows directly from the comparison principle applied
to $u$ and $-u.$ The remaining assertions follow directly from the well-known
regularity result by Lieberman \cite[Theorem 1]{Lieb} applied to the
$p$-Laplacian by taking into account that (\ref{bt}) and (\ref{uinf}) yield
$\left\vert u\right\vert \leq M_{0}.$
\end{proof}

\begin{remark}
According to \cite[Theorem 1]{Lieb} the constants $\eta_{p}$ and $K_{p}$
depend only on $\tau,$ $N,$ $p,$ $M_{0}$ and $\Omega.$
\end{remark}

As for the constant $k_{p}$ defined in (\ref{kp}) we observe that%
\[
0<\left\Vert \nabla\phi_{p}\right\Vert _{\infty}\leq k_{p}\leq K_{p},\text{
\ whenever }p>1,
\]
where the latter inequality follows from Proposition \ref{p1}.

Using the well-known fact (see \cite{ARMA99}):%
\[
\left\Vert v\right\Vert _{\infty}\leq\left\Vert d_{\Omega}\right\Vert
_{\infty}\left\Vert \nabla v\right\Vert _{\infty}\quad\forall\,v\in
W^{1,\infty}(\Omega)\cap C_{0}(\overline{\Omega}),
\]
one obtains a lower bound to $k_{p}$ is in terms of $\left\Vert \phi
_{p}\right\Vert _{\infty}$ and $\left\Vert d_{\Omega}\right\Vert _{\infty}$:%

\begin{equation}
\frac{\left\Vert \phi_{p}\right\Vert _{\infty}}{\left\Vert d_{\Omega
}\right\Vert _{\infty}}\leq\left\Vert \nabla\phi_{p}\right\Vert _{\infty}\leq
k_{p},\text{ \ whenever }p>1. \label{k-}%
\end{equation}

\begin{corollary}
\label{p2}If $g\in L^{\infty}(\Omega)$ and $u\in W_{0}^{1,p}(\Omega)$ is the
only weak solution to
\[
\left\{
\begin{array}
[c]{rrll}%
-\Delta_{p}v & = & g & \text{in }\Omega\\
v & = & 0 & \text{on }\partial\Omega,
\end{array}
\right.
\]
then $u\in C^{1,\eta_{p}}(\overline{\Omega})$ and the following estimates hold%
\begin{equation}
\left\vert u(x)\right\vert \leq\left\Vert g\right\Vert _{\infty}^{1/(p-1)}%
\phi_{p}(x)\quad\forall\,x\in\overline{\Omega}, \label{uinfa}%
\end{equation}%
\begin{equation}
\left\vert u\right\vert _{1,\eta_{p}}\leq K_{p}\left\Vert g\right\Vert
_{\infty}^{1/(p-1)}, \label{grad1a}%
\end{equation}
and%
\begin{equation}
\left\Vert \nabla u\right\Vert _{\infty}\leq k_{p}\left\Vert g\right\Vert
_{\infty}^{1/(p-1)}. \label{grad2}%
\end{equation}

\end{corollary}

\begin{proof}
It suffices to consider $\left\Vert g\right\Vert _{\infty}\not =0.$ Let
\[
v:=\frac{u}{\left\Vert g\right\Vert _{\infty}^{1/(p-1)}}\text{ \ and
\ }\widetilde{g}:=\frac{g}{\left\Vert g\right\Vert _{\infty}}.
\]
As $u\in\mathcal{S}_{p}$ (note that $\left\Vert \widetilde{g}\right\Vert
_{\infty}=1$) the estimates (\ref{uinfa}) and (\ref{grad1a}) follow directly
from Proposition \ref{p1} applied to $v$, and (\ref{grad2}) follows from the
definition of $k_{p}.$
\end{proof}

\begin{lemma}
One has
\begin{equation}
1\leq\lambda_{p}^{1/(p-1)}\left\Vert \phi_{p}\right\Vert _{\infty}.
\label{lbep}%
\end{equation}

\end{lemma}

\begin{proof}
As
\[
-\Delta_{p}e_{p}=\lambda_{p}e_{p}^{p-1}\leq\lambda_{p}=-\Delta_{p}(\lambda
_{p}^{1/(p-1)}\phi_{p})\text{ in }\Omega
\]
and $e_{p}=\phi_{p}=0$ on $\partial\Omega,$ it follows from the comparison
principle that%
\[
e_{p}\leq\lambda_{p}^{1/(p-1)}\phi_{p}\text{ in }\Omega.
\]
Hence, $1=\left\Vert e_{p}\right\Vert _{\infty}\leq\lambda_{p}^{1/(p-1)}%
\left\Vert \phi_{p}\right\Vert _{\infty}.$
\end{proof}

Now, we present our main existence result.

\begin{proof}
[Proof of Theorem \ref{sbSP}]We recall, from the definition of $\mathcal{E}%
(M)$ that
\begin{equation}
\lambda\frac{A_{p}}{M^{p-q}}+\beta\frac{B_{p}}{M^{p-(a+b)}}+m\frac{A_{p}%
}{M^{p-l}e^{-\alpha M^{s}}}\leq1. \label{M}%
\end{equation}
Let us consider the closed, convex and bounded subset $F\subset C^{1}%
(\overline{\Omega})$ defined by%
\begin{equation}
F:=\left\{  u\in C^{1}(\overline{\Omega}):\left(  \frac{\lambda}{\lambda_{p}%
}\right)  ^{1/(p-q)}e_{p}\leq u\leq\frac{M}{\left\Vert \phi_{p}\right\Vert
_{\infty}}\phi_{p}\text{ and \ }\left\Vert \nabla u\right\Vert _{\infty}%
\leq\frac{k_{p}M}{\left\Vert \phi_{p}\right\Vert _{\infty}}\right\}  .
\label{F}%
\end{equation}

Let $T:F\subset C^{1}(\overline{\Omega})\rightarrow C^{1}(\overline{\Omega})$
be the operator that assigns to each $u\in F$ the only function $T(u)\in
W_{0}^{1,p}(\Omega)$ satisfying%
\begin{equation}
\left\{
\begin{array}
[c]{rrll}%
-\Delta_{p}(T(u)) & = & \lambda(T(u))^{q-1}+\beta u^{a-1}\left\vert \nabla
u\right\vert ^{b}+mu^{l-1}e^{\alpha u^{s}} & \text{in }\Omega\\
T(u) & = & 0 & \text{on }\partial\Omega.
\end{array}
\right.  \label{Ua}%
\end{equation}
Thus, for each $u\in F$ the function $U=T(u)$ is the only weak solution in
$W_{0}^{1,p}(\Omega)$ to the Dirichlet problem%
\begin{equation}
\left\{
\begin{array}
[c]{rrll}%
-\Delta_{p}v & = & g(x,v) & \text{in }\Omega\\
v & = & 0 & \text{on }\partial\Omega,
\end{array}
\right.  \label{U}%
\end{equation}
where the nonlinearity $g(x,t)$ is defined from $u$ by the expression%
\[
g(x,t):=\lambda t^{q-1}+\beta u(x)^{a-1}\left\vert \nabla u(x)\right\vert
^{b}+mu(x)^{l-1}e^{\alpha u(x)^{s}},\text{ }x\in\overline{\Omega}\text{ \ and
}t\geq0.
\]

The uniqueness of $U$ follows from \cite{DS} as $g$ is sublinear in the
variable $t$ (recall that $p>q$).

Let us define
\[
\overline{u}:=\frac{M}{\left\Vert \phi_{p}\right\Vert _{\infty}}\phi_{p}\text{
\ and \ }\underline{u}:=\left(  \frac{\lambda}{\lambda_{p}}\right)
^{1/(p-q)}e_{p}.
\]
We are going to prove the existence of $U$ from the sub-super solution method
by showing that: $\overline{u}$ is a supersolution to \ref{U}, $\underline{u}$
is a subsolution to the same problem, and $\underline{u}\leq\overline{u}$ \ in
$\Omega.$

As\ $u\in F$ we have that
\[
0\leq u\leq\frac{M}{\left\Vert \phi_{p}\right\Vert _{\infty}}\phi_{p}\leq
M\text{ \ in }\Omega
\]
and
\[
0\leq\left\vert \nabla u\right\vert \leq\left\Vert \nabla u\right\Vert
_{\infty}\leq\frac{k_{p}M}{\left\Vert \phi_{p}\right\Vert _{\infty}}\text{
\ in }\Omega.
\]

Hence,
\[
0\leq u^{l-1}e^{\alpha u^{s}}\leq M^{l-1}e^{\alpha M^{s}}\text{ \ in }\Omega.
\]

Therefore, as
\[
0\leq\overline{u}:=\frac{M}{\left\Vert \phi_{p}\right\Vert _{\infty}}\phi
_{p}\leq M\text{ \ in }\Omega
\]
the above estimates and (\ref{M}) imply that
\begin{equation}
g(x,\overline{u}(x))\leq\left(  \frac{M}{\left\Vert \phi_{p}\right\Vert
_{\infty}}\right)  ^{p-1}\quad\forall\,x\in\Omega\label{sup1}%
\end{equation}
since
\begin{align*}
g(x,\overline{u}(x))  &  =\lambda\overline{u}(x)^{q-1}+\beta u(x)^{a-1}%
\left\vert \nabla u(x)\right\vert ^{b}+mu(x)^{l-1}e^{\alpha u(x)^{s}}\\
&  \leq\lambda M^{q-1}+\beta M^{a-1}\left(  \frac{k_{p}M}{\left\Vert \phi
_{p}\right\Vert _{\infty}}\right)  ^{b}+mM^{l-1}e^{\alpha M^{s}}\leq\left(
\frac{M}{\left\Vert \phi_{p}\right\Vert _{\infty}}\right)  ^{p-1}.
\end{align*}

As $-\Delta_{p}\phi_{p}=1$ in $\Omega,$ it follows from (\ref{sup1}) that
\[
-\Delta_{p}\overline{u}=\left(  \frac{M}{\left\Vert \phi_{p}\right\Vert
_{\infty}}\right)  ^{p-1}(-\Delta_{p}\phi_{p})=\left(  \frac{M}{\left\Vert
\phi_{p}\right\Vert _{\infty}}\right)  ^{p-1}\geq g(x,\overline{u})\text{ in
}\Omega.
\]
Thus, recalling that $\overline{u}=0$ on $\partial\Omega,$ we conclude that
$\overline{u}$ is a supersolution to (\ref{U}).

Using that $e_{p}^{p-q}\leq\left\Vert e_{p}\right\Vert _{\infty}^{p-q}\leq1$
in $\Omega,$ we have that
\begin{align*}
-\Delta_{p}\underline{u}  &  =\left(  \frac{\lambda}{\lambda_{p}}\right)
^{(p-1)/(p-q)}\left(  -\Delta_{p}e_{p}\right) \\
&  =\left(  \frac{\lambda}{\lambda_{p}}\right)  ^{(p-1)/(p-q)}\lambda_{p}%
e_{p}^{p-1}\\
&  =\left(  \frac{\lambda}{\lambda_{p}}\right)  \left(  \frac{\lambda}%
{\lambda_{p}}\right)  ^{(q-1)/(p-q)}\lambda_{p}e_{p}^{p-q}e_{p}^{q-1}\\
&  \leq\lambda\left(  \frac{\lambda}{\lambda_{p}}\right)  ^{(q-1)/(p-q)}%
e_{p}^{q-1}=\lambda\underline{u}^{q-1}\leq g(x,\underline{u})\text{ \ in
}\Omega.
\end{align*}
Hence, as $\underline{u}=0$ on $\partial\Omega$ we conclude that
$\underline{u}$ is a subsolution to (\ref{U}).

In order to prove that $\underline{u}\leq\overline{u}$ \ in $\Omega$ we first
observe from (\ref{ApBp}) and (\ref{M}) that%
\[
\lambda\frac{\left\Vert \phi_{p}\right\Vert _{\infty}^{p-1}}{M^{p-q}}\leq1,
\]
so that%
\[
\frac{M}{\left\Vert \phi_{p}\right\Vert _{\infty}}\geq\left(  \lambda
\left\Vert \phi_{p}\right\Vert _{\infty}^{q-1}\right)  ^{1/(p-q)}=\left(
\frac{\lambda}{\lambda_{p}}\right)  ^{1/(p-q)}\lambda_{p}^{1/(p-q)}\left\Vert
\phi_{p}\right\Vert _{\infty}^{(q-1)/(p-q)}.
\]

Hence, by using (\ref{lbep}) we obtain%
\begin{equation}
\left(  \frac{M}{\left\Vert \phi_{p}\right\Vert _{\infty}}\right)  ^{p-1}%
\geq\left(  \frac{\lambda}{\lambda_{p}}\right)  ^{(p-1)/(p-q)}\lambda_{p}
\label{aux8}%
\end{equation}
since
\[
\lambda_{p}^{1/(p-q)}\left\Vert \phi_{p}\right\Vert _{\infty}^{(q-1)/(p-q)}%
\geq\lambda_{p}^{1/(p-q)}\left(  \lambda_{p}^{-1/(p-1)}\right)  ^{(q-1)/(p-q)}%
=\lambda_{p}^{1/(p-1)}.
\]

It follows from (\ref{aux8}) that
\begin{align*}
-\Delta_{p}\underline{u}  &  =\left(  \frac{\lambda}{\lambda_{p}}\right)
^{(p-1)/(p-q)}\lambda_{p}e_{p}^{p-1}\\
&  \leq\left(  \frac{M}{\left\Vert \phi_{p}\right\Vert _{\infty}}\right)
^{p-1}e_{p}^{p-1}\leq\left(  \frac{M}{\left\Vert \phi_{p}\right\Vert _{\infty
}}\right)  ^{p-1}=-\Delta_{p}\overline{u}%
\end{align*}
and this implies that $\underline{u}\leq\overline{u}$ \ in $\Omega$ by the
comparison principle.

Therefore, we can apply the sub-super solution method to guarantee the
existence of a weak solution $U\in W_{0}^{1,p}(\Omega)$ to (\ref{U})
satisfying
\begin{equation}
\underline{u}=\left(  \frac{\lambda}{\lambda_{p}}\right)  ^{1/(p-q)}e_{p}\leq
U\leq\frac{M}{\left\Vert \phi_{p}\right\Vert _{\infty}}\phi_{p}=\overline
{u}\text{ \ in }\Omega. \label{F1}%
\end{equation}

As%
\[
\left\Vert g(x,U)\right\Vert _{\infty}\leq\lambda M^{q-1}+\beta M^{a-1}\left(
\frac{k_{p}M}{\left\Vert \phi_{p}\right\Vert _{\infty}}\right)  ^{b}%
+mM^{l-1}e^{\alpha M^{s}}\leq\left(  \frac{M}{\left\Vert \phi_{p}\right\Vert
_{\infty}}\right)  ^{p-1}%
\]
we note from Corollary \ref{p2} that $U=T(u)\in C^{1,\eta_{p}}(\overline
{\Omega}),$%
\[
\left\vert U\right\vert _{1,\eta_{p}}\leq K_{p}\frac{M}{\left\Vert \phi
_{p}\right\Vert _{\infty}}%
\]
and
\begin{equation}
\left\Vert \nabla U\right\Vert _{\infty}\leq k_{p}\frac{M}{\left\Vert \phi
_{p}\right\Vert _{\infty}}. \label{F2}%
\end{equation}

Combining (\ref{F1}) and (\ref{F2}) we conclude that $U=T(u)\in F,$ meaning
that $T(F)\subset F.$

Using the compactness of the embedding $C^{1,\eta_{p}}(\overline{\Omega
})\hookrightarrow C^{1}(\overline{\Omega})$ we can verify that $T:F\rightarrow
F$ is compact. Therefore, Schauder's fixed point theorem guarantees the
existence of a fixed point $u_{p}\in F.$ Consequently, $T(u_{p})=u_{p}=u$ in
(\ref{Ua}), so that
\[
\left\{
\begin{array}
[c]{rrll}%
-\Delta_{p}u_{p} & = & \lambda u_{p}^{q-1}+\beta u_{p}^{a-1}\left\vert \nabla
u_{p}\right\vert ^{b}+mu_{p}^{l-1}e^{\alpha u_{p}^{s}} & \text{in }\Omega\\
u_{p} & = & 0 & \text{on }\partial\Omega
\end{array}
\right.
\]
in the weak sense. In addition, as $u_{p}\in F$ the estimates (\ref{bounds}) hold.
\end{proof}

\begin{remark}
We have improved the lower bound in (\ref{F}) with respect to \cite{BE} (which
was also used in \cite{AM23}) since we have shown in (\ref{aux8}) that%
\[
\left(  \frac{\lambda}{\lambda_{p}}\right)  ^{1/(p-q)}=\min\left\{  \left(
\frac{\lambda}{\lambda_{p}}\right)  ^{1/(p-q)},\,\frac{M}{\left\Vert \phi
_{p}\right\Vert _{\infty}\lambda_{p}^{1/(p-1)}}\right\}  .
\]

\end{remark}

As a simple application of Theorem \ref{sbSP} we obtain the following
existence result.

\begin{corollary}
\label{up}Assume that $\partial\Omega\in C^{1,\tau}$ for some $\tau\in(0,1).$
Let $q,$ $a,$ $b,$ $l,$ $\lambda$ and $\beta$ be fixed, with $q,a,l\geq1,$
$b,\lambda,\beta>0,$ and $s,\alpha\geq0.$ For each
\[
p>\max\left\{  q,a+b,l\right\}
\]
there exists a positive constant $M_{p}$ satisfying
\begin{equation}
\lambda\frac{A_{p}}{M_{p}^{p-q}}+\beta\frac{B_{p}}{M_{p}^{p-(a+b)}}=\frac
{1}{2}. \label{Mp}%
\end{equation}
Moreover, if%
\begin{equation}
0<m\leq m_{p}:=\frac{M_{p}^{p-l}}{2A_{p}e^{\alpha M_{p}^{s}}}, \label{m<mp}%
\end{equation}
then the problem (\ref{P}) admits a weak solution $u_{p}\in W_{0}^{1,p}%
(\Omega)$ satisfying (\ref{bounds}).
\end{corollary}

\begin{proof}
The hypotheses imply that the function
\[
\varphi(t):=\lambda\frac{A_{p}}{t^{p-q}}+\beta\frac{B_{p}}{t^{p-(a+b)}},\text{
\ }t>0,
\]
satisfies $\lim_{t\rightarrow0^{+}}\varphi(t)=+\infty$ and $\lim
_{t\rightarrow\infty}\varphi(t)=0^{+}.$ Consequently, there exists $M_{p}>0$
such that $\varphi(M_{p})=\frac{1}{2},$ which is (\ref{Mp}). Hence, as $m$
satisfies (\ref{m<mp}) we have%
\[
\frac{\lambda A_{p}}{M_{p}^{p-q}}+\frac{\beta B_{p}}{M_{p}^{p-(a+b)}}%
+\frac{mA_{p}}{M_{p}^{p-l}e^{-\alpha M^{s}}}\leq\frac{1}{2}+\frac{m_{p}A_{p}%
}{M_{p}^{p-l}e^{-\alpha M^{s}}}=1,
\]
so that $(\lambda,\beta,m)\in\mathcal{E}(M_{p}).$
\end{proof}

Now, we present some more applications of Theorem \ref{sbSP} that extends or
complements some recent results for problems involving exponential and
convection terms.

\begin{corollary}
\label{cor1}Assume that $\partial\Omega\in C^{1,\tau}$ for some $\tau
\in(0,1).$ Let $q,$ $a,$ $b,$ $l,$ $\alpha,$ $s\ $ and $m$ be fixed, with
$p>q\geq1,$ $m,b>0,$ $\alpha,s\geq0,$ $a\geq1,$ and
\[
l>p\geq a+b.
\]
There exists a positive constant $M_{p}$ such that if%
\begin{equation}
(\lambda,\beta)\in\mathcal{D}:=\left\{  (\lambda,\beta)\in\mathbb{R}_{+}%
^{2}:\lambda\frac{A_{p}}{M_{p}^{p-q}}+\beta\frac{B_{p}}{M_{p}^{p-(a+b)}}%
\leq\frac{1}{2}\right\}  , \label{D}%
\end{equation}
then the problem (\ref{P}) admits a weak solution $u_{p}\in W_{0}^{1,p}%
(\Omega)$ satisfying (\ref{bounds}).
\end{corollary}

\begin{proof}
We can write the inequality (\ref{M}) as%
\begin{equation}
\varphi_{1}(M)+\varphi_{2}(M)\leq1 \label{1}%
\end{equation}
where
\[
\varphi_{1}(t):=\lambda\frac{A_{p}}{t^{p-q}}+\beta\frac{B_{p}}{t^{p-(a+b)}%
},\text{ \ }t>0
\]
and%
\[
\varphi_{2}(t):=mA_{p}t^{l-p}e^{\alpha t^{s}},\text{ \ }t>0.
\]

As $\varphi_{2}$ is strictly increasing and%
\[
\lim_{t\rightarrow0}\varphi_{2}(t)=0\text{ \ and \ }\lim_{t\rightarrow\infty
}\varphi_{2}(t)=\infty,
\]
there exists a unique $M_{p}>0$ such that%
\begin{equation}
\varphi_{2}(M_{p})=\frac{1}{2}. \label{1a}%
\end{equation}

Using such $M_{p}$ we define $\mathcal{D}$ in (\ref{D}) by the inequality
\begin{equation}
\varphi_{1}(M_{p})\leq\frac{1}{2}. \label{1b}%
\end{equation}
Thus, if $(\lambda,\beta)\in\mathcal{D}$ we obtain (\ref{M}) from (\ref{1}),
(\ref{1a}) and (\ref{1b}). The existence result follows then from Theorem
\ref{sbSP}.
\end{proof}

Geometrically, $\mathcal{D}$ is the region in the quadrant $\mathbb{R}_{+}%
^{2}$ of the $\lambda\beta$-plane that lies below the line%
\[
\lambda\frac{A_{p}}{M_{p}^{p-q}}+\beta B_{p}M_{p}^{a+b-p}=\frac{1}{2}.
\]

In \cite{AL}, de Araujo and Faria considered the Dirichlet problem \
\begin{equation}
\left\{
\begin{array}
[c]{rrll}%
-\Delta_{N}u & = & \gamma(a_{1}u^{r_{1}}+a_{2}\left\vert \nabla u\right\vert
^{r_{2}})+f(u) & \text{in }\Omega,\\
u & > & 0 & \text{in }\Omega,\\
u & = & 0 & \text{on }\partial\Omega
\end{array}
\right.  \label{P2}%
\end{equation}
where $0<r_{1},r_{2}<N-1,$ $a_{1}>0,$ $a_{2}\geq0,$ and $f:[0,\infty
)\rightarrow\mathbb{R}$ is a continuous function satisfying%
\[
0\leq f(t)\leq a_{3}t^{r_{3}}e^{\alpha t^{N/(N-1)}},\text{ \ where }%
a_{3},\alpha>0\text{ and }r_{3}>N-1.
\]
They used an approximation scheme to prove the existence of a weak solution
$u\in W_{0}^{1,N}(\Omega)$ to (\ref{P2}) whenever $\gamma\in(0,\gamma^{\ast
}),$ for some $\gamma^{\ast}>0.$

Note that (\ref{P2}) is a particular case of (\ref{P}) with $\lambda=\gamma
a_{1},$ \ $\beta=\gamma a_{2},$ \ $q=r_{1}+1,$ \ $a=1,$ \ $b=r_{2},$
\ $m=a_{3},$ \ $l=r_{3}+1,$ \ and $s=\frac{N}{N-1}.$

We remark that Corollary \ref{cor1} extends the result by de Araujo and Faria
in \cite{AL} (for the case $f(t)=a_{3}t^{r_{3}}e^{\alpha t^{N/(N-1)}}$) since
it admits $p\not =N$ and also allows the convection term to be multiplied by a
power of the solution.

The following corollary further extends the result of de Araujo and Faria (for
the case $f(t)=a_{3}t^{r_{3}}e^{\alpha t^{N/(N-1)}}$) by admitting $a+b>p$ (in
(\ref{P2}) this means that $r_{2}>N-1$).

\begin{corollary}
\label{cor2}Assume that $\partial\Omega\in C^{1,\tau}$ for some $\tau
\in(0,1).$ Let $q,$ $a,$ $b,$ $l,$ $\alpha,$ $s,$ $\beta$ and $m$ be fixed,
with $a,l\geq1,$ $\alpha,s\geq0$ and $\beta,m>0.$ Suppose that%
\[
1\leq q<p<\min\left\{  a+b,l\right\}  .
\]
There exists a positive constant $M_{p}$ such that if
\[
0<\lambda\leq\lambda^{\ast}:=\frac{M_{p}^{p-q}}{2A_{p}}%
\]
then (\ref{P}) admits a weak solution $u_{p}\in W_{0}^{1,p}(\Omega)$
satisfying (\ref{bounds}).
\end{corollary}

\begin{proof}
Now, we write the inequality (\ref{M}) as%
\[
\varphi_{1}(M)+\varphi_{2}(M)\leq1
\]
where
\[
\varphi_{1}(t):=\lambda\frac{A_{p}}{t^{p-q}},\text{ \ }t>0
\]
and%
\[
\varphi_{2}(t):=\beta B_{p}t^{(a+b)-p}+mA_{p}t^{l-p}e^{\alpha t^{s}},\text{
\ }t>0.
\]

As $\varphi_{2}$ is strictly increasing and%
\[
\lim_{t\rightarrow0}\varphi_{2}(t)=0\text{ \ and \ }\lim_{t\rightarrow\infty
}\varphi_{2}(t)=\infty,
\]
there exists $M_{p}>0$ such that
\[
\varphi_{2}(M_{p})=\frac{1}{2}.
\]

Thus, if $\lambda\leq\lambda^{\ast}$ then $\varphi_{1}(M_{p})\leq\frac{1}{2}$
and (\ref{M}) holds. Consequently, we can apply Theorem \ref{sbSP} to arrive
at the desired result.
\end{proof}

In the notation of (\ref{P2}) we have
\[
\gamma\leq\gamma^{\ast}:=\frac{M_{p}^{p-q}}{2a_{1}A_{p}}%
\]
where $M$ is defined by the equation%
\[
\gamma a_{2}B_{p}M_{p}^{(a+b)-p}+mA_{p}M_{p}^{l-p}e^{\alpha M^{s}}=\frac{1}%
{2}.
\]

Proceeding as in the two previous proofs we obtain the following result.

\begin{corollary}
\label{cor3}Assume that $\partial\Omega\in C^{1,\tau}$ for some $\tau
\in(0,1).$ Let $q,$ $a,$ $b,$ $l,$ $\alpha,$ $s\ $ and $\beta$ be fixed, with
$b,\beta>0,$ $a\geq1,$ $\alpha,s\geq0.$ Suppose that
\[
1\leq q<l=p<a+b.
\]
If
\[
M_{p}:=\left(  \frac{1}{2\beta B_{p}}\right)  ^{\frac{1}{(a+b)-p}}%
\]
and
\[
(\lambda,m)\in\mathcal{D}:=\left\{  (\lambda,m)\in\mathbb{R}_{+}%
\times\mathbb{R}_{+}:\lambda\frac{A_{p}}{M_{p}^{p-q}}+mA_{p}e^{\alpha
M_{p}^{s}}\leq\frac{1}{2}\right\}  ,
\]
then (\ref{P}) admits a weak solution $u_{p}\in W_{0}^{1,p}(\Omega)$
satisfying (\ref{bounds}).
\end{corollary}

\section{A nonexistence result\label{none}}

In \cite{Peral} Garcia Azorero and Peral Alonso proved in Theorem 2.1 that the
problem
\begin{equation}
\left\{
\begin{array}
[c]{rrll}%
-\Delta_{p}u & = & me^{u} & \text{in }\Omega\\
u & = & 0 & \text{on }\partial\Omega,
\end{array}
\right.  \label{lane1}%
\end{equation}
does not have a solution if
\begin{equation}
m>\max\left\{  \lambda_{p},\lambda_{p}\left(  \frac{p-1}{e}\right)
^{p-1}\right\}  . \label{m>}%
\end{equation}

In this section, we extends the nonexistence result by Garcia Azorero and
Peral Alonso for the more general equation%
\[
-\Delta_{p}u=mu^{l-1}e^{\alpha u^{s}}+g(x)\text{ in }\Omega.
\]

The following lemma was proved by Allegretto and Huang (see \cite[Theorem
2.4]{AH}) as a consequence of Picone's identity.

\begin{lemma}
\label{picone}Let $h\in L^{\infty}(\Omega)$ be a nonnegative function. The
Dirichlet problem%
\[
\left\{
\begin{array}
[c]{lll}%
-\Delta_{p}u=\lambda_{p}\left\vert u\right\vert ^{p-2}u+h(x) & \text{in} &
\Omega\\
u\geq0 & \text{on} & \partial\Omega
\end{array}
\right.
\]
has a weak solution if and only if $h\equiv0$ in $\Omega$ and $u=0$ on
$\partial\Omega.$ In this case, the solution is a multiple of $e_{p}.$
\end{lemma}

\begin{proposition}
\label{bm}Suppose that $u_{m}\in W_{0}^{1.p}(\Omega)$ is a positive weak
solution to the Dirichlet problem
\begin{equation}
\left\{
\begin{array}
[c]{rrll}%
-\Delta_{p}u & = & mu^{l-1}e^{\alpha u^{s}}+g(x) & \text{in }\Omega\\
u & = & 0 & \text{on }\partial\Omega,
\end{array}
\right.  \label{lane}%
\end{equation}
where $g\in L^{\infty}(\Omega),$ $g\geq0,$ $1\leq l<p$ and $\alpha,s>0.$ Then%
\begin{equation}
m<\lambda_{p}\left(  \frac{p-l}{\alpha se}\right)  ^{\frac{p-l}{s}}.
\label{mcp}%
\end{equation}

\end{proposition}

\begin{proof}
Let us consider the strictly positive function%
\[
Q(t):=m\frac{e^{\alpha t^{s}}}{t^{p-l}},\text{ \ }t>0.
\]

A simple calculation shows that the only critical point of $Q$ is
\[
t_{m}=\left(  \frac{p-l}{\alpha s}\right)  ^{1/s}.
\]

As%
\[
\lim_{t\rightarrow0^{+}}Q(t)=\lim_{t\rightarrow+\infty}Q(t)=+\infty,
\]
we have that $t_{m}$ is the only global minimum point. Thus,
\[
Q(t)>Q(t_{m})=mC_{p}\quad\forall\,t\not =t_{m},
\]
where%
\[
C_{p}:=\left(  \frac{\alpha se}{p-l}\right)  ^{\frac{p-l}{s}}.
\]

It follows that%
\begin{equation}
mu_{m}^{l-1}e^{\alpha u_{m}^{s}}=Q(u_{m})u_{m}^{p-1}\geq Q(t_{m})u_{m}%
^{p-1}=mC_{p}u_{m}^{p-1}\text{ in }\Omega\label{meq2}%
\end{equation}
with the equality occurring only if either $u_{m}=t_{m}$ or $u_{m}=0.$

Now, we observe that
\begin{equation}
-\Delta_{p}u_{m}=\lambda_{p}u_{m}^{p-1}+h\text{ in }\Omega\label{meq}%
\end{equation}
where%
\[
h:=mu_{m}^{l-1}e^{\alpha u_{m}^{s}}+g-\lambda_{p}u_{m}^{p-1}\geq mu_{m}%
^{l-1}e^{\alpha u_{m}^{s}}-\lambda_{p}u_{m}^{p-1}\geq(mC_{p}-\lambda_{p}%
)u_{m}^{p-1}.
\]

We are going to show that
\[
mC_{p}-\lambda_{p}<0
\]
which is (\ref{mcp}). Let us suppose, by contradiction, that
\[
mC_{p}\geq\lambda_{p}.
\]
Owing to (\ref{meq}) and Lemma \ref{picone} this implies that $h\equiv0$ a.e.
in $\Omega.$ Thus,%
\begin{equation}
0\leq(mC_{p}-\lambda_{p})u_{m}^{p-1}\leq mu_{m}^{l-1}e^{\alpha u_{m}^{s}%
}+g-\lambda_{p}u_{m}^{p-1}=0. \label{meq1}%
\end{equation}

Hence, if $mC_{p}>\lambda_{p}$ then (\ref{meq1}) leads to the absurd
\[
u_{m}=0\text{ a.e. in }\Omega,
\]
and if $mC_{p}=\lambda_{p}$, then (\ref{meq2}) and (\ref{meq1}) yield%
\[
g=\lambda_{p}u_{m}^{p-1}-mu_{m}^{l-1}e^{\alpha u_{m}^{s}}\leq0.
\]
This implies that $g=0$ and leads to the equality in (\ref{meq2}) which is
absurd, for it means that
\[
u_{m}=t_{m}\text{ a.e. in }\Omega.
\]

\end{proof}

We remark that (\ref{mcp}) improves the estimate (\ref{m>}) when $1<p\leq1+e.$
In fact, in this case, (\ref{lane1}) also has no solution if
\[
\lambda_{p}\left(  \frac{p-1}{e}\right)  ^{p-1}\leq m\leq\lambda_{p}.
\]

\section{Asymptotic behavior\label{sec2}}

In this section we assume a stronger assumption on the regularity of $\Omega$:
either $\partial\Omega\in C^{1,1}$ or $\Omega$ convex and $\partial\Omega\in
C^{1,\tau}$ ($\tau$ as before).

Our goal is to prove the uniform convergence of $u_{p}$ to $d_{\Omega},$ as
$p\rightarrow\infty,$ where $u_{p}\in W_{0}^{1,p}(\Omega)$ is the solution to
(\ref{P}) given by Corollary \ref{up}. Thus, we consider: $p>\max\left\{
q,a+b,l\right\}  ,$ the positive parameters $\lambda$ and $\beta$ arbitrary,
and the parameter $m$ restricted to the interval $(0,m_{p}].$ We recall that
\[
m_{p}:=\frac{M_{p}^{p-l}}{2A_{p}e^{\alpha M_{p}^{s}}},
\]
$M_{p}$ is defined by (\ref{Mp}), and%
\begin{equation}
\left(  \frac{\lambda}{\lambda_{p}}\right)  ^{1/(p-q)}e_{p}\leq u_{p}\leq
\frac{M_{p}}{\left\Vert \phi_{p}\right\Vert _{\infty}}\phi_{p}\text{ \ and
\ }\left\vert \nabla u_{p}\right\vert \leq\frac{k_{p}M_{p}}{\left\Vert
\phi_{p}\right\Vert _{\infty}}. \label{boundsup}%
\end{equation}

To achieve our goal we will make use of the explicit gradient estimates
derived by Ercole in \cite{GE}, They are based on the results by Cianchi and
Maz'ya in \cite{CM11} for a class of operators that includes the $p$-Laplacian
as a very particular case.

We recall that the Lorentz space $L^{\theta,1}(\Omega)$ consists of all
measurable functions $v:\Omega\rightarrow\mathbb{R}$ such that%
\[
\int_{0}^{\left\vert \Omega\right\vert }s^{-1/\theta^{\prime}}\left\vert
v^{\ast}(s)\right\vert \mathrm{d}s<\infty.
\]
Here, $\theta^{\prime}=\frac{\theta}{\theta-1}$ and $v^{\ast}:[0,\infty
)\rightarrow\lbrack0,\infty]$ stands for the decreasing rearrangement of $v$,
which is defined as%
\[
v^{\ast}(s):=\left\{
\begin{array}
[c]{lll}%
\sup\left\{  t\geq0:\mu_{v}(t)>s\right\}  & \text{if} & 0\leq s\leq\left\vert
\Omega\right\vert \\
0 & \text{if} & s>\left\vert \Omega\right\vert ,
\end{array}
\right.
\]
where
\[
\mu_{v}(t):=\left\vert \left\{  x\in\Omega:v(x)>t\right\}  \right\vert ,\quad
t\geq0,
\]
is the distribution function of $v.$

As it is well known, $L^{\theta,1}(\Omega)$ is a Banach space endowed with the
norm
\[
\left\Vert v\right\Vert _{\theta,1}:=\int_{0}^{\left\vert \Omega\right\vert
}\left\vert v^{\ast\ast}(s)\right\vert s^{-1/\theta^{\prime}}\mathrm{d}s
\]
where $v^{\ast\ast}:(0,\infty)\rightarrow\lbrack0,\infty)$ is defined as%
\[
v^{\ast\ast}(s):=\frac{1}{s}\int_{0}^{s}v^{\ast}(r)\mathrm{d}r,\quad s>0.
\]
Thus, if $g\in L^{\infty}(\Omega),$ then
\begin{equation}
\left\Vert g\right\Vert _{N,1}\leq\left\Vert g^{\ast}\right\Vert _{\infty}%
\int_{0}^{\left\vert \Omega\right\vert }s^{-1/N^{\prime}}\mathrm{d}%
s=N\left\vert \Omega\right\vert ^{\frac{1}{N}}\left\Vert g\right\Vert
_{\infty} \label{gN1}%
\end{equation}
as $\left\Vert g^{\ast}\right\Vert _{\infty}=\left\Vert g\right\Vert _{\infty
}.$

\begin{lemma}
\label{ge}Suppose that $p\geq2$ and either $\partial\Omega\in C^{1,1}$ or
$\Omega$ convex. Let $v\in W_{0}^{1,p}(\Omega)$ be the solution of the
Dirichlet problem
\[
\left\{
\begin{array}
[c]{rrll}%
-\Delta_{p}v & = & g & \text{in }\Omega\\
v & = & 0 & \text{on }\partial\Omega,
\end{array}
\right.
\]
where $g\in L^{\infty}(\Omega).$ Then, there exist positive constants $c$ and
$\gamma,$ that are uniform with respect to $p$ and $g,$ such that
\begin{equation}
\left\Vert \nabla v\right\Vert _{\infty}^{p-1}\leq cp^{\gamma}\left\Vert
g\right\Vert _{\infty}. \label{ge2}%
\end{equation}

\end{lemma}

\begin{proof}
According to Theorem 1.2 of \cite{GE},
\begin{equation}
\left\Vert \nabla v\right\Vert _{\infty}^{p-1}\leq Cp^{(\frac{5}{2}-\frac
{2}{p})+\frac{\theta N}{\theta-(N-1)}}\left\Vert g\right\Vert _{N,1},
\label{ge1}%
\end{equation}
where $C$ is a positive constant that depends at most on $N$ and $\Omega.$
This estimate holds under the following assumptions: $N\geq3,$ $\partial
\Omega\in W^{2}L^{\theta,1},$ for some $\theta>N-1,$ and $f\in L^{N,1}%
(\Omega).$ Moreover, if the assumption $\partial\Omega\in W^{2}L^{\theta,1}$
is replaced with $\Omega$ convex, then the estimate (\ref{ge1}) writes as%
\[
\left\Vert \nabla v\right\Vert _{\infty}^{p-1}\leq Cp^{(\frac{5}{2}-\frac
{2}{p})}\left\Vert g\right\Vert _{N,1}.
\]

Since $C^{1,1}\subset W^{2}L^{\theta,1},$ (\ref{gN1}) and (\ref{ge1}) lead to
(\ref{ge2}) with $c:=CN\left\vert \Omega\right\vert ^{\frac{1}{N}}$ and
$\gamma=\frac{5}{2}+\frac{\theta N}{\theta-(N-1)}.$ If $\Omega$ is convex we
can take $\gamma=\frac{5}{2}.$

As for $N=2,$ by assuming that $\partial\Omega\in W^{2}L^{\theta,1},$ for some
$\theta>1,$ and $f\in L^{q}(\Omega)$, for some $q>2,$ Theorem 1.3 of \cite{GE}
yields the estimate
\begin{equation}
\left\Vert \nabla v\right\Vert _{\infty}^{p-1}\leq Cp^{(\frac{5}{2}-\frac
{2}{p})+\frac{2\theta}{\theta-1}}\left\Vert g\right\Vert _{q} \label{ge3}%
\end{equation}
where $C$ depends at most on $\Omega$ and $q.$ Hence, as $\left\Vert
g\right\Vert _{q}\leq\left\Vert g\right\Vert _{\infty}\left\vert
\Omega\right\vert ^{1/p}$ and $C^{1,1}\subset W^{2}L^{\theta,1}$ the estimate
(\ref{ge3}) holds with $c:=C\left\vert \Omega\right\vert ^{1/p}$ and
$\gamma:=\frac{5}{2}+\frac{2\theta}{\theta-1}.$ If $\Omega$ is convex, then
(\ref{ge3}) writes as%
\[
\left\Vert \nabla v\right\Vert _{\infty}^{p-1}\leq Cp^{(\frac{5}{2}-\frac
{2}{p})}\left\Vert g\right\Vert _{q}%
\]
in which case we can take $\gamma=\frac{5}{2}.$
\end{proof}

\begin{remark}
Following Cianchi and Maz'ya in \cite{CM11}, the assumption $\partial\Omega\in
W^{2}L^{\theta,1}$ means that the boundary of $\Omega$ is locally the subgraph
of a function of $N-1$ variables whose second-order distributional derivatives
lie on the Lorentz space $L^{\theta,1}.$ The regularity hypothesis
$\partial\Omega\in W^{2}L^{N-1,1}$ is the weakest possible integrability
assumption on second-order derivatives for the first order derivatives to be
continuous, and hence for $\partial\Omega\in C^{1,0}$ \cite{CP98}.
\end{remark}

In the sequel we will use some known results that are gathered in the
following lemma.

\begin{lemma}
\label{known}The following convergence results are well known:
\end{lemma}

\begin{enumerate}
\item $\phi_{p}$ converges uniformly in $\overline{\Omega}$ to $d_{\Omega}$ as
$p\rightarrow\infty$ (see \cite{BDM, K1}).

\item $\lim_{p\rightarrow\infty}\lambda_{p}^{1/p}=\left\Vert d_{\Omega
}\right\Vert _{\infty}^{-1}$ (see \cite{ARMA99}).

\item For each sequence $\left(  p_{n}\right)  ,$ with $p_{n}\rightarrow
\infty,$ there exists a subsequence $\left(  p_{n_{j}}\right)  $ and a
function $e_{\infty}\in W^{1,\infty}(\Omega)\cap C_{0}(\overline{\Omega})$
such that: $\left\Vert e_{\infty}\right\Vert _{\infty}=1,$ $e_{p_{n_{j}}}$
converges uniformly in $\overline{\Omega}$ to $e_{\infty},$ and
\[
0<e_{\infty}\leq\frac{d_{\Omega}}{\left\Vert d_{\Omega}\right\Vert _{\infty}%
}\text{ in }\overline{\Omega}\text{ (see \cite{ARMA99}).}%
\]

\end{enumerate}

\begin{remark}
\label{square}The strict positiveness of $e_{\infty}$ follows from the Harnack
inequality proved in \cite[Theorem 1]{Bhatta} (see also \cite[Corollary
4.5]{ML}) \ since $e_{\infty}$ is $\infty$-superharmonic and not identically
zero ($\left\Vert e_{\infty}\right\Vert _{\infty}=1$). The equality
$e_{\infty}=\frac{d_{\Omega}}{\left\Vert d_{\Omega}\right\Vert _{\infty}}$
does not hold for a general bounded domain $\Omega.$ It holds for balls,
annuli\ and stadiums (see \cite{Yu}), but not for a square, for example (see
\cite[Proposition 4.1]{ARMA99}).
\end{remark}

The following result is crucial in our analysis.

\begin{proposition}
One has%
\begin{equation}
\lim_{p\rightarrow\infty}k_{p}=1. \label{kinf}%
\end{equation}

\end{proposition}

\begin{proof}
We observe from (\ref{k-}) and item $1$ of Lemma \ref{known} that%
\[
1=\lim_{p\rightarrow\infty}\frac{\left\Vert \phi_{p}\right\Vert _{\infty}%
}{\left\Vert d_{\Omega}\right\Vert _{\infty}}\leq\liminf_{p\rightarrow\infty
}k_{p}.
\]

According to Lemma \ref{ge}%
\[
\left\Vert \nabla w\right\Vert _{\infty}\leq(cp^{\gamma})^{\frac{1}{p-1}%
}\text{\ \ for all }w\in\mathcal{S}_{p},\text{ }p\geq2.
\]
Consequently,
\[
k_{p}\leq(cp^{\gamma})^{\frac{1}{p-1}}\text{\ \ for all }p\geq2,
\]
so that%
\[
\limsup_{p\rightarrow\infty}k_{p}\leq\lim_{p\rightarrow\infty}(cp^{\gamma
})^{\frac{1}{p-1}}=1.
\]

\end{proof}

\begin{lemma}
\label{Mpto}One has%
\begin{equation}
\lim_{p\rightarrow\infty}M_{p}=\left\Vert d_{\Omega}\right\Vert _{\infty},
\label{liMp}%
\end{equation}%
\begin{equation}
\lim_{p\rightarrow\infty}\left(  \frac{\left\Vert \phi_{p}\right\Vert
_{\infty}}{M_{p}}\right)  ^{p}=\frac{1}{2(\lambda\left\Vert d_{\Omega
}\right\Vert _{\infty}^{q-1}+\beta\left\Vert d_{\Omega}\right\Vert _{\infty
}^{a-1})} \label{liMp1}%
\end{equation}
and%
\begin{equation}
\lim_{p\rightarrow\infty}m_{p}=m_{\infty}, \label{limp}%
\end{equation}
where $m_{\infty}$ is defined in (\ref{minf}).
\end{lemma}

\begin{proof}
We can write (\ref{Mp}) as
\begin{equation}
\lambda\left\Vert \phi_{p}\right\Vert _{\infty}^{q-1}\left(  \frac{\left\Vert
\phi_{p}\right\Vert _{\infty}}{M_{p}}\right)  ^{p-q}+\beta\left\Vert \phi
_{p}\right\Vert _{\infty}^{a-1}k_{p}^{b}\left(  \frac{\left\Vert \phi
_{p}\right\Vert _{\infty}}{M_{p}}\right)  ^{p-r}=\frac{1}{2} \label{Mp1}%
\end{equation}
where%
\[
r:=a+b.
\]

It follows from (\ref{Mp1}) that
\[
\frac{\left\Vert \phi_{p}\right\Vert _{\infty}}{M_{p}}\leq\frac{1}%
{(2\lambda\left\Vert \phi_{p}\right\Vert _{\infty}^{q-1})^{1/(p-q)}}%
\]
so that%
\begin{equation}
\limsup_{p\rightarrow\infty}\frac{\left\Vert \phi_{p}\right\Vert _{\infty}%
}{M_{p}}\leq\lim_{p\rightarrow\infty}\frac{1}{(2\lambda\left\Vert \phi
_{p}\right\Vert _{\infty}^{q-1})^{1/(p-q)}}=1. \label{ls2}%
\end{equation}
Here we have used the fact that
\begin{equation}
\lim_{p\rightarrow\infty}\left\Vert \phi_{p}\right\Vert _{\infty}=\left\Vert
d_{\Omega}\right\Vert _{\infty} \label{kw}%
\end{equation}
(according to item 1 from Lemma \ref{known}).

Without loss of generality we analyze the case $r\geq q$ (the case $r<q$ is
analogous), so that
\[
\left(  \frac{\left\Vert \phi_{p}\right\Vert _{\infty}}{M_{p}}\right)
^{p-q}\leq\left(  \frac{\left\Vert \phi_{p}\right\Vert _{\infty}}{M_{p}%
}\right)  ^{p-r}%
\]
for all $p$ sufficiently large. Hence, (\ref{Mp1}) yields%
\[
\frac{1}{2}\leq\left(  \lambda\left\Vert \phi_{p}\right\Vert _{\infty}%
^{q-1}+\beta\left\Vert \phi_{p}\right\Vert _{\infty}^{a-1}k_{p}^{b}\right)
\left(  \frac{\left\Vert \phi_{p}\right\Vert _{\infty}}{M_{p}}\right)  ^{p-r}%
\]
for all $p$ sufficiently large. Then, using (\ref{kinf}) and (\ref{kw}) we
make $p\rightarrow\infty$ in the inequality
\[
\left(  \frac{1}{2}\right)  ^{\frac{1}{p-r}}\left(  \lambda\left\Vert \phi
_{p}\right\Vert _{\infty}^{q-1}+\beta\left\Vert \phi_{p}\right\Vert _{\infty
}^{a-1}k_{p}^{b}\right)  ^{-\frac{1}{p-r}}\leq\frac{\left\Vert \phi
_{p}\right\Vert _{\infty}}{M_{p}}%
\]
to find
\begin{equation}
1\leq\liminf_{p\rightarrow\infty}\frac{\left\Vert \phi_{p}\right\Vert
_{\infty}}{M_{p}}. \label{li2}%
\end{equation}

Combining (\ref{ls2}) and (\ref{li2}) we conclude that
\begin{equation}
\lim_{p\rightarrow\infty}\frac{\left\Vert \phi_{p}\right\Vert _{\infty}}%
{M_{p}}=1 \label{Mp2}%
\end{equation}
and then, in view of (\ref{kw}), we obtain (\ref{liMp}).

Now, let us set
\[
L:=\lim_{p\rightarrow\infty}\left(  \frac{\left\Vert \phi_{p}\right\Vert
_{\infty}}{M_{p}}\right)  ^{p}.
\]
Combining (\ref{kinf}), (\ref{kw}) and (\ref{Mp2}) we obtain from (\ref{Mp1})
the equality%
\[
\lambda\left\Vert d_{\Omega}\right\Vert _{\infty}^{q-1}L+\beta\left\Vert
d_{\Omega}\right\Vert _{\infty}^{a-1}L=\frac{1}{2},
\]
which leads to (\ref{liMp1}).

Finally, after noticing that%

\[
m_{p}=\frac{M_{p}^{p-l}}{2A_{p}e^{\alpha M_{p}^{s}}}=\frac{1}{2}\left(
\frac{M_{p}}{\left\Vert \phi_{p}\right\Vert _{\infty}}\right)  ^{p-1}%
M_{p}^{1-l}\frac{1}{e^{\alpha M_{p}^{s}}}%
\]
we obtain (\ref{limp}) from (\ref{liMp}) and (\ref{liMp1}).
\end{proof}

\begin{lemma}
\label{uconv}If $p_{n}\rightarrow\infty,$ then there exists a subsequence
$(u_{p_{n_{j}}})$ converging uniformly in $\overline{\Omega}$ to a function
$u_{\infty}\in W^{1,\infty}(\Omega)\cap C_{0}(\overline{\Omega})$ such that
\begin{equation}
\left\Vert d_{\Omega}\right\Vert _{\infty}e_{\infty}\leq u_{\infty}\leq
d_{\Omega}\text{ \ in }\overline{\Omega}\text{ \ } \label{infbounds}%
\end{equation}
where $e_{\infty}$ is a positive $\infty$-superharmonic function satisfying
$\left\Vert e_{\infty}\right\Vert _{\infty}=1.$
\end{lemma}

\begin{proof}
Combining Lemma \ref{known} with (\ref{boundsup}), (\ref{kinf}) and
(\ref{liMp}) we have that%
\[
\lim_{p\rightarrow\infty}\left\Vert u_{p}\right\Vert _{\infty}=d_{\Omega
}\text{ \ and \ }\limsup_{p\rightarrow\infty}\left\Vert \nabla u_{p}%
\right\Vert _{\infty}\leq1\text{ \ in }\overline{\Omega}.
\]
Therefore, by Arzel\'{a}-Ascoli Theorem there exists a subsequence
$(u_{p_{n_{j}}})$ converging uniformly in $\overline{\Omega}$ to a function
$u_{\infty}\in W^{1,\infty}(\Omega)\cap C_{0}(\overline{\Omega}).$ By item 3
of Lemma \ref{known} we can assume that $e_{p_{n_{j}}}$ converges uniformly to
a positive $\infty$-superharmonic function satisfying $\left\Vert e_{\infty
}\right\Vert _{\infty}=1.$

Hence, taking into account item 2 of Lemma \ref{known}, the inequalities in
(\ref{infbounds}) follow after letting $j\rightarrow\infty$ in the estimates%
\[
\left(  \frac{\lambda}{\lambda_{p_{n_{j}}}}\right)  ^{1/(p_{n_{j}}%
-q)}e_{p_{n_{j}}}\leq u_{p_{n_{j}}}\leq\frac{M_{p_{n_{j}}}}{\left\Vert
\phi_{p_{n_{j}}}\right\Vert _{\infty}}\phi_{p_{n_{j}}}\text{.}%
\]

\end{proof}

\begin{proposition}
\label{main2}Assume that $\partial\Omega\in C^{1,1}.$ Let $q,$ $a,$ $b,$ $l,$
$\lambda$ and $\beta$ be fixed, with $q>1,$ $a,l\geq1,$ and $b,s,\alpha
,\lambda,\beta>0.$ For each $p>\max\left\{  q,a+b,l\right\}  $ let $M_{p},$
$m_{p}$ and $u_{p}\in W_{0}^{1,p}(\Omega)$ be as in Corollary \ref{up}. Then,%
\[
\lim_{p\rightarrow\infty}u_{p}=d_{\Omega}\text{ uniformly in }\overline
{\Omega}.
\]

\end{proposition}

\begin{proof}
It follows from Lemma \ref{uconv} that, up to subsequence, $u_{p}$ converges
uniformly in $\overline{\Omega}$ to a function $u_{\infty}\in W^{1,\infty
}(\Omega)\cap C_{0}(\overline{\Omega})$ satisfying (\ref{infbounds}).

We recall that $u_{p}$ is also the only weak solution to the Dirichlet problem%
\begin{equation}
\left\{
\begin{array}
[c]{rrll}%
-\Delta_{p}u & = & \lambda\left\vert u\right\vert ^{q-2}u+h_{p} & \text{in
}\Omega\\
u & > & 0 & \text{in }\Omega,\\
u & = & 0 & \text{on }\partial\Omega
\end{array}
\right.  \label{dirih}%
\end{equation}
where%
\[
h_{p}:=\beta u_{p}^{a-1}\left\vert \nabla u_{p}\right\vert ^{b}+mu_{p}%
^{l-1}e^{\alpha u_{p}^{s}}.
\]

As
\[
0\leq h_{p}\leq M_{p}^{a-1}\left(  \frac{k_{p}M_{p}}{\left\Vert \phi
_{p}\right\Vert _{\infty}}\right)  ^{b}+mM_{p}^{l-1}e^{\alpha M_{p}^{s}}%
\]
we note that
\begin{equation}
\limsup_{p\rightarrow\infty}\left\Vert h_{p}\right\Vert _{\infty}\leq
\beta\left\Vert d_{\Omega}\right\Vert _{\infty}^{a-1}+m\left\Vert d_{\Omega
}\right\Vert _{\infty}^{l-1}e^{\alpha\left\Vert d_{\Omega}\right\Vert
_{\infty}^{s}}<\infty. \label{aux2}%
\end{equation}

We also know that the solution to (\ref{dirih}) is the only positive minimizer
of the functional%
\[
I_{p}(v)=\frac{1}{p}\left\Vert \nabla v\right\Vert _{p}^{p}-\frac{1}{q}%
\int_{\Omega}\left\vert v\right\vert ^{q}\mathrm{d}x-\int_{\Omega}%
h_{p}v\mathrm{d}x,\text{ \ }v\in W_{0}^{1,p}(\Omega).
\]
Hence, recalling that $d_{\Omega}\in W_{0}^{1,p}(\Omega)$ and that $\left\vert
\nabla d_{\Omega}\right\vert =1$ a.e. in $\Omega,$ we obtain from the
inequality $I_{p}(u_{p})\leq I_{p}(d_{\Omega})$ that%
\begin{equation}
\frac{1}{q}\int_{\Omega}(d_{\Omega}^{q}-u_{p}^{q})\mathrm{d}x+\int_{\Omega
}h_{p}(d_{\Omega}-u_{p})\mathrm{d}x\leq\frac{\left\vert \Omega\right\vert }%
{p}-\frac{1}{p}\left\Vert \nabla u_{p}\right\Vert _{p}^{p}\leq\frac{\left\vert
\Omega\right\vert }{p}. \label{aux1}%
\end{equation}

It follows from (\ref{aux1}) that%
\begin{equation}
\limsup_{p\rightarrow\infty}\left[  \frac{1}{q}\int_{\Omega}(d_{\Omega}%
^{q}-u_{p}^{q})\mathrm{d}x+\int_{\Omega}h_{p}(d_{\Omega}-u_{p})\mathrm{d}%
x\right]  \leq0. \label{aux3}%
\end{equation}

As $d_{\Omega}-u_{\infty}\geq0$, the uniform convergence from $u_{p}$ to
$u_{\infty}$ implies that%
\begin{equation}
\lim_{p\rightarrow\infty}\int_{\Omega}(d_{\Omega}^{q}-u_{p}^{q})\mathrm{d}%
x=\int_{\Omega}(d_{\Omega}^{q}-u_{\infty}^{q})\mathrm{d}x\geq0. \label{aux6}%
\end{equation}

Using again that $d_{\Omega}-u_{\infty}\geq0$ we have
\begin{equation}
\int_{\Omega}h_{p}(d_{\Omega}-u_{p})\mathrm{d}x=\int_{\Omega}h_{p}(d_{\Omega
}-u_{\infty})\mathrm{d}x+\int_{\Omega}h_{p}(u_{\infty}-u_{p})\mathrm{d}%
x\geq\int_{\Omega}h_{p}(u_{\infty}-u_{p})\mathrm{d}x. \label{aux7}%
\end{equation}

The uniform convergence from $u_{p}$ to $u_{\infty}$ combined with and
(\ref{aux2}) yields%
\[
\lim_{p\rightarrow\infty}\int_{\Omega}h_{p}(u_{\infty}-u_{p})\mathrm{d}x=0
\]
since
\[
\left\vert \int_{\Omega}h_{p}(u_{\infty}-u_{p})\mathrm{d}x\right\vert
\leq\left\Vert u_{\infty}-u_{p}\right\Vert _{\infty}\left\Vert h_{p}%
\right\Vert _{\infty}\left\vert \Omega\right\vert .
\]

Thus, it follows from (\ref{aux7}) that
\begin{equation}
\liminf_{p\rightarrow\infty}\int_{\Omega}h_{p}(d_{\Omega}-u_{p})\mathrm{d}%
x\geq0. \label{aux5}%
\end{equation}

Hence,%
\begin{equation}
\liminf_{p\rightarrow\infty}\left[  \frac{1}{q}\int_{\Omega}(d_{\Omega}%
^{q}-u_{p}^{q})\mathrm{d}x+\int_{\Omega}h_{p}(d_{\Omega}-u_{p})\mathrm{d}%
x\right]  \geq0. \label{aux4}%
\end{equation}

Combining (\ref{aux3}) and (\ref{aux4}) we obtain%
\[
\lim_{p\rightarrow\infty}\left[  \frac{1}{q}\int_{\Omega}(d_{\Omega}^{q}%
-u_{p}^{q})\mathrm{d}x+\int_{\Omega}h_{p}(d_{\Omega}-u_{p})\mathrm{d}x\right]
=0.
\]

In view of (\ref{aux6}) and (\ref{aux5}) we arrive at%
\[
\frac{1}{q}\int_{\Omega}(d_{\Omega}^{q}-u_{\infty}^{q})\mathrm{d}x=0.
\]

Therefore, using once more that $u_{\infty}\leq d_{\Omega}$ we conclude that
$u_{\infty}=d_{\Omega}.$

Observing that the limit function is always $d_{\Omega}$ we conclude that
$u_{p}$ converges uniformly to $d_{\Omega}$ in $\overline{\Omega}$
(independently of subsequences).
\end{proof}

\begin{proof}
[Proof of Theorem \ref{main}]Let $p_{0}>\max\left\{  q,a+b,l\right\}  $ be
such that $0<m<m_{p}$ for all $p>p_{0}.$ This follows by combining
(\ref{limp}) with the fact that $m<m_{\infty}.$ Thus, if $p>p_{0}$ the
existence of $u_{p}$ follows from Corollary \ref{up} and the convergence
(\ref{upd}) follows from Proposition \ref{main2}.
\end{proof}

\section*{Acknowledgments}

Anderson L. A. de Araujo was partially supported by FAPEMIG/Brazil
APQ-02375-21, RED-00133-21 and by CNPq/Brazil 307575/2019-5. Grey Ercole was
partially supported by FAPEMIG/Brazil PPM-00137-18, CNPq/Brazil 305578/2020-0
and FAPDF 04/2021.


\begin{thebibliography}{99}                                                                                               %
\bibitem {AH}W. Allegreto, Y.X. Huang: A Picone's identity for the
$p$-Laplacian and applications, Nonlinear Anal. \textbf{32} (1998) 819--830.

\bibitem {AL}A.L.A de Araujo, L.F.O. Faria: Positive solutions of quasilinear
elliptic equations with exponential nonlinearity combined with convection
term, J. Differential Equations \textbf{26}7 (2019) 4589--4608.

\bibitem {AM23}A.L.A de Araujo, M. Montenegro: Existence of solution for a
nonlinear equation with supercritical exponential growth, J. Fixed Point
Theory Appl. (2023) 25:26.

\bibitem {Bhatta}T. Bhattacharya: An elementary proof of the Harnack
inequality for non-negative infinity-superharmonic functions, Electron. J.
Differential Equations \textbf{2001.44} (2001) 1--8.

\bibitem {BDM}T. Bhattacharya, E. DiBenedetto and J. Manfredi: Limits as
$p\rightarrow\infty$ of $\Delta_{p}u_{p}=f$ and related extremal problems,
Rend. Sem. Mat. Univ. Pol. Torin Fascicolo Speciale (1989) 15--68.

\bibitem {Bo17}M. Bocea, M. Mih\u{a}ilescu: On a family of inhomogeneous
torsional creep problems, Proc. Amer. Math. Soc. \textbf{145} (2017) 4397--4409.

\bibitem {BE}H. Bueno, G. Ercole: A quasilinear problem with fast growing
gradient, Applied Mathematics Letters \textbf{26} (2013) 520--523.



\bibitem {CP98}A. Cianchi, L. Pick: Sobolev embeddings into $BMO$, $VMO$ and
$L^{\infty},$ Ark. Math. \textbf{36} (1998) 317--340.

\bibitem {CM11}A. Cianchi, V. G. Maz'ya: Global Lipschitz regularity for a
class of quasilinear elliptic equations. Comm. Partial Differential Equations,
\textbf{36 }(2011) 100--133.



\bibitem {Djairo}D.G. de Figueiredo, J.P. Gossez, , H.R. Quoirin, P. Ubilla:
Elliptic equations involving the $p$-Laplacian and a gradient term having
natural growth, Rev. Mat. Iberoam. \textbf{35} (2019) 173--194.

\bibitem {DS}J.I. D\'{\i}az, J.E. Saa: Existence et unicit\'{e} de solutions
positives pour certaines \'{e}quations elliptiques quasilin\'{e}aires, C.R.
Acad. Sci., Paris \textbf{305} (S\'{e}rie I) (1987) 521--524.

\bibitem {GE}G. Ercole: On a global gradient estimate in $p$-Laplacian
problems, submitted
(\href{https://doi.org/10.48550/arXiv.2302.05538}{https://doi.org/10.48550/arXiv.2302.05538}%
).

\bibitem {GE1}G. Ercole: On a family of problems driven by rapidly growing
operators, Monatsh. Math. (2023)
\href{https://link.springer.com/article/10.1007/s00605-023-01844-z}{DOI
10.1007/s00605-023-01844-z}

\bibitem {EFMP}G. Ercole, G.M. Figueiredo, V.M. Magalh\~{a}es and G.A.
Pereira: The limiting behavior of global minimizers in non-reflexive
Orlicz-Sobolev spaces, Proc. Amer. Math. Soc. \textbf{150} (2022) 5267--5280.

\bibitem {FM19}M. F\u{a}rc\u{a}\c{s}eanu, M. Mih\u{a}ilescu: On a family of
torsional creep problems involving rapidly growing operators in divergence
form, Proc. Roy. Soc. Edinburgh Sect. A \textbf{149} (2019) 495--510.

\bibitem {FMD17}M. F\u{a}rc\u{a}\c{s}eanu, M. Mih\u{a}ilescu, D.
Stancu-Dumitru: On the convergence of the sequence of solutions for a family
of eigenvalue problems, Math. Methods Appl. Sci. \textbf{40} (2017) 6919-6926.

\bibitem {Peral}J. Garcia Azorero, I. Peral Alonso: On an Emden-Fowler type
equation, Nonlinear Anal. \textbf{18} (1992) 1085--1097.

\bibitem {GD21}A. Grecu, D. Stancu-Dumitru: The asymptotic behavior of
solutions to a class of inhomogeneous problems: an Orlicz--Sobolev space
approach, Electron. J. Qual. Theory Differ. Equ. \textbf{38} (2021) 1--20.



\bibitem {ARMA99}J. Juutine, P. Lindqvist and J. Manfredi: The $\infty
$-eigenvalue problem, Arch. Ration. Mech. Anal. \textbf{148} (1999) 89--105.

\bibitem {K1}B. Kawohl: On a family of torsional creep problems, J. Reine
Angew. Math. \textbf{410} (1990) 1--22.

\bibitem {Lieb}G.M. Lieberman: Boundary regularity for solutions of degenerate
elliptic equations, Nonlinear Anal. \textbf{12} (1988) 1203--1219.

\bibitem {ML}J. Manfredi, P. Lindqvist: Note on $\infty$-superharmonic
functions, Rev. Mat. Univ. Complut. Madrid \textbf{10} (1997) 471--480.

\bibitem {Mi18}M. Mih\u{a}ilescu, D. Stancu-Dumitru and C. Varga: The
convergence of nonnegative solutions for the family of problems $-\Delta
_{p}u=\lambda e^{u}$ as $p\rightarrow\infty$, ESAIM Control Optim. Calc. Var.
\textbf{24} (2018) 569--578.

\bibitem {Yu}Y. Yu : Some properties of the ground states of the infinity
Laplacian, Indiana Univ. Math. J. \textbf{56} (2007) 947--964 .
\end{thebibliography}
\end{document}